\documentclass[12pt,leqno]{article}
\usepackage{amsmath, amsthm}
\usepackage{amssymb}
\usepackage{amscd}
\usepackage{url}
\usepackage{graphicx}
\usepackage{color}

\numberwithin{equation}{section}

\theoremstyle{plain}

\newtheorem{thm}[subsection]{Theorem}
\newtheorem{prop}[subsection]{Proposition}
\newtheorem{lem}[subsection]{Lemma}

\theoremstyle{definition}

\newcommand{\Z}{\mathbb Z}

\providecommand{\keywords}[1]
{
  \noindent\small	
  {\textit{Keywords}:} #1
}

\providecommand{\subjclass}[1]
{
  \noindent\small	
  {\textit{Mathematics Subject Classification} (2020):} #1
}

\title{A simple evaluation of a theta value and the Kronecker limit formula}
\author{Fernando Chamizo\thanks{The author is partially supported by the  PID2020-113350GB-I00 grant of the MICINN (Spain) and by ``Severo Ochoa Programme for Centres of Excellence in R\&{D}'' (SEV-2015-0554).}
}

\date{July 18, 2021}

\begin{document}

\maketitle

%

%
%

\begin{abstract}
We evaluate the classic sum $\sum_{n\in\Z} e^{-\pi n^2}$. The novelty of our approach is that it does not require any prior knowledge about modular forms, elliptic functions or analytic continuations. Even the $\Gamma$ function, in terms of which the result is expressed, only appears as a complex function in the computation of a real integral by the residue theorem. Another contribution of this note is to provide a very simple proof of the Kronecker limit formula. 
\end{abstract}



\maketitle

\section{Introduction}

Our goal is to give a proof of the following result with very few prerequisites. In general, the special values of theta and allied functions are related to deep topics in number theory (complex multiplication, class field theory, modular forms, elliptic functions, etc., cf. \cite{cox}, \cite{ChRa}) which we avoid here. 

\begin{thm}\label{main}
 Consider $\theta(z)= \sum_{n\in\Z} e^{\pi i n^2 z}$. Then
 \[
  \theta(i)=(2\pi)^{-1/4}
  \sqrt{\frac{\Gamma(1/4)}{\Gamma(3/4)}}
  =
  \frac{\Gamma(1/4)}{\pi^{3/4}\sqrt{2}}
 \]
 with $\Gamma$ the classical Gamma function $\Gamma(s)=\int_0^\infty t^{s-1}e^{-t}\; dt$. 
\end{thm}

We will prove the first equality, the second follows from the relation $\Gamma(s)\Gamma(1-s)=\pi\csc(\pi s)$ that do not use elsewhere. In fact the $\Gamma$ function only appears as a complex function in the computation of an integral (Lemma~\ref{integral}) and, beyond that, we barely use its defining integral representation for $s>1$.
\smallskip

Except for a special case of the Jacobi triple product identity and the well known formula for the number of representations as a sum of two squares (both separated in \S\ref{auxres} and admitting elementary proofs, not included here), the proof is completely self-contained. The techniques only involve basic real and complex variable methods. 
No modular properties of $\theta$ and $\eta$ and no functional equations of any $L$-function or Eisenstein series nor their analytic continuations,  are required. 

Our argument includes a proof of a version of the (first) Kronecker limit formula (Proposition~\ref{kronecker}) simpler than the ones we have found in the literature (cf. \cite{motohashi}) which may have independent interest. We address the reader to the interesting paper \cite{DuImTo} for the history and relevance of this formula. 

\section{Two auxiliary results}\label{auxres}

We first recall the factorization of the $\theta$ function. 

\begin{lem}\label{triple}
 For $|q|<1$,
 \[
  \sum_{n=-\infty} ^\infty 
  q^{n^2}
  =
  \prod_{n=1}^\infty
  \big(1-q^{2n}\big)\big(1+q^{2n-1}\big)^2.
 \]
\end{lem}

The next result is the classic formula for $r(n)$, the number of representations of $n$ as a sum of two squares, in terms of the nontrivial character $\chi$ modulo~$4$ (i.e., $\chi(n)=(-1)^{(n-1)/2}$ for $n$ odd and zero for $n$ even).
\begin{lem}\label{rn}
 For $n\in\Z^+$ and $s>1$, we have 
 \[
  r(n)=4\sum_{d\mid n}\chi(n)
  \quad\text{or equivalently,}\quad
  \sum_{n=1}^\infty r(n)n^{-s}=4\zeta(s)L(s)
 \]
 with $\zeta(s)=\sum_{n=1}^\infty n^{-s}$ the Riemann zeta function and $L(s)=\sum_{n=1}^\infty \chi(n) n^{-s}$. 
\end{lem}

We will say some words about their proofs. 

Lemma~\ref{triple} comes from the Jacobi triple product identity which admits elementary combinatorial proofs (see \cite[\S8.3]{hua} and \cite{andrews}) but arguably, even today, the conceptually most enlightening  proof is the classic one based on complex analysis 
\cite[\S10.1]{StSh}. It uses the invariance by two translations of certain entire function to conclude that it is a constant, which is computed with a beautiful argument due to Gauss \cite[\S78]{rademacher}. 

Lemma~\ref{rn} can be derived from the triviality of some spaces of modular forms or from some properties of elliptic functions 
\cite[\S10.3.1]{StSh},
\cite[\S84]{rademacher}.
A less demanding proof, requiring quadratic residues and almost nothing else,  is to use the representations of an integer by the quadratic forms in a class
\cite[\S12.4]{hua}.
A longer alternative is 
to show that $\Z[i]$ is a UFD and deduce the result from 
$r(p)=4(1+\chi(p))$ for $p$ prime, which is essentially Fermat two squares theorem \cite[Art.182]{gauss} (see \cite{zagier} for a ``one-sentence'' proof of the latter).

\section{The Kronecker limit formula and the theta evaluation}

We first state a compact version of the Kronecker limit formula and provide a proof only requiring the residue theorem and 
the very easy \cite[p.23]{iwaniec} and well known result
$(s-1)\zeta(s)\to 1$ as $s\to 1^+$. 

The Epstein zeta function $\zeta(s,Q)$ associated to a positive definite binary quadratic form $Q$ and the Dedekind $\eta$ function are defined by:
\[
 \zeta(s,Q) 
 = 
 \sum_{\vec{n}\in\Z^2\setminus \{\vec{0}\}}
 \big(Q(\vec{n})\big)^{-s}
 \quad\text{ and }\quad
 \eta(z)=
 e^{\pi i z/12}
 \prod_{n=1}^\infty \big(1-e^{2\pi i nz}\big).
\]
We assume $s>1$ and $\Im z>0$ to assure the convergence.

\begin{prop}\label{kronecker}
 Let $Q(x,y)=ax^2+bxy+cy^2$  be a real form with $D=4ac-b^2>0$ and $a>0$. Then 
 \[
  \lim_{s\to 1^+}
  \Big(
  \frac{\sqrt{D}}{4\pi}\zeta(s,Q)
  -
  \zeta(2s-1)
  \Big)
  =
  \log\frac{\sqrt{a/D}}{|\eta(z_Q)|^2}
  \qquad
  \text{with}\quad
  z_Q = 
  \frac{-b+i\sqrt{D}}{2a}.
 \]
\end{prop}

\begin{proof}
 Let $f(s)=-\int_{-\infty}^\infty Q(x,1)^{-s}\; dx$. The limit in the statement equals $L_1-L_2$ with 
 \[
  L_1
  =
  \lim_{s\to 1^+}
  \frac{\sqrt{D}}{4\pi}
  \Big(
  \zeta(s,Q)+2\zeta(2s-1)f(s)
  \Big),
  \ 
  L_2
  =
  \lim_{s\to 1^+}
  \zeta(2s-1)\Big(
  \frac{\sqrt{D}}{2\pi}
  f(s)
  +1
  \Big).
 \]
 L'H\^opital's rule shows $L_2=\frac{\sqrt{D}}{4\pi}f'(1)$ because $(2s-2)\zeta(2s-1)\to 1$ (and the residue theorem assures $f(1)=-2\pi/\sqrt{D}$). Then the result follows if we prove
 \begin{equation}\label{goal1}
  L_1=
  -\log|\eta(z_Q)|^2
  \qquad\text{and}\qquad
  f'(1)=-\frac{4\pi}{\sqrt{D}}\log\sqrt{\frac{a}{D}}.
 \end{equation}
 We have $f'(1)=\int_{-\infty}^\infty (\log p)/p$ with $p(x)=ax^2+bx+c$. With the change of variables 
 $2a x+b=\sqrt{D}\tan(t/2)$ we obtain
 \[
  f'(1)
  =
  -
  \frac{2}{\sqrt{D}}
  \int_{-\pi}^\pi
  \log\frac{2|\cos(t/2)|}{\sqrt{D/a}}
  =
  -
  \frac{2}{\sqrt{D}}
  \Re
  \int_{C}
  \log\Big(\frac{1+z}{\sqrt{D/a}}\Big)
  \frac{dz}{iz}
 \]
 with $C$ the unit circle, 
 where we have used $\log \big(2|\cos(t/2)|\big) =\Re \log(1+z)$ with $z=e^{it}$. 
 Cauchy's integral formula gives the second identity in \eqref{goal1}. 
 \medskip
 
 When we sum $Q(m,n)^{-s}$ the contribution of $n=0$ is $2a^{-s}\zeta(2s)$. Let  
 $g_s(z)=Q(z,1)^{-s}+Q(z,-1)^{-s}$.
 For $n\ne 0$, $Q(m,n)=n^2Q(m/n,1)=n^2Q(-m/n,-1)$.
 Then by the residue theorem in the band 
 $B_\epsilon=\{|\Im z|<\epsilon\}$
 with $0<\epsilon <\Im z_Q$,
 \[
  \zeta(s,Q)-2\frac{\zeta(2s)}{a^s}
  =
  \sum_{n=1}^\infty
  \frac{1}{n^{2s}}
  \sum_{m\in\Z} g_s\big(\frac mn\big)
  =
  \sum_{n=1}^\infty
  \frac{-1}{2n^{2s-1}}
  \int_{\partial B_\epsilon}
  g_s(z)i\cot(\pi nz)\; dz.
 \]
 As $g_s$ is even, $\int_{\partial B_\epsilon}=-2\int_{L_\epsilon}$
 with $L_\epsilon = \{\Im z=\epsilon\}$ oriented to the right
 and the sum is $\sum_{n}n^{1-2s}\int_{L_\epsilon}$. Note that $\int_{L_\epsilon}g_s=\int_{L_0}g_s=-2f(s)$. Then adding $2\zeta(2s-1)f(s)$ 
 is equivalent to replace $i\cot(\pi nz)$ by $i\cot(\pi nz)-1$ in $\int_{L_\epsilon}$.
 The expansion 
 $i\cot w-1=2e^{2iw}/(1-e^{2iw})=2(e^{2iw}+e^{4iw}+\dots)$ 
 assures an exponential decay and we have
 \[
  L_1=
  \frac{\sqrt{D}}{4\pi}
  \Big(
  2\frac{\zeta(2)}{a}
  +
  \sum_{n,k=1}^\infty
  \frac{2}{n}
  \int_{L_\epsilon}g_1(z)e^{2\pi inkz}\; dz
  \Big).
 \]
 Substitute $\zeta(2)=\pi^2/6$ and 
 note that 
 $g_1(z)=\big(a(z-z_Q)(z-\bar{z}_Q)\big)^{-1}+\big(a(z+z_Q)(z+\bar{z}_Q)\big)^{-1}$. The residue theorem in $\{\Im z>\epsilon\}$ gives promptly
 \[ 
  L_1
  =
  \frac{\pi\sqrt{D}}{12a}
  +
  \sum_{n,k=1}^\infty
  \frac{1}{n}
  \big(
  e^{2\pi nkiz_Q}+e^{-2\pi nki\bar{z}_Q}
  \big)
  =
  \frac{\pi\sqrt{D}}{12a}
  -
  \sum_{k=1}^\infty
  \log \big|1- e^{2\pi kiz_Q}\big|^2
 \]
 where the second equality comes from $\log(1-w)+\log(1-\bar{w})= \log|1-w|^2$. The sum is 
 $\log\big(|\eta(z_Q)|^2|e^{-\pi i z_Q/6}|\big)$
 and the proof of \eqref{goal1} is complete. 
\end{proof}

The evaluation of an integral will be play a role in the final step of our proof of Theorem~\ref{main}. We proceed again employing the residue theorem.
\begin{lem}\label{integral}
 Let
 \[
  I = 
  \frac{1}{\pi}
  \int_0^\infty
  \frac{\log t}{\cosh t}\; dt
  \qquad\text{then}\quad
  \exp(I)=
  \frac{\Gamma(3/4)}{\Gamma(1/4)}\sqrt{2\pi}.
 \]
\end{lem}

\begin{proof}
 Consider $f(z)=i\sec(2\pi z)\log\Gamma(1/2+z)$ on the vertical band $B = \big\{|\Re z |<1/2\big\}$. 
 It defines a meromorphic function (for certain branch of the logarithm because $\Gamma$ does not vanish) with simple poles at $z_{\pm}=\pm 1/4$. Clearly the residues satisfy $2\pi i\text{Res}(f,z_{\pm})=\pm \log\Gamma(1/2+z_{\pm})$. 
 This function is integrable along $\partial B$ and the residue theorem shows
 \[
  \log\frac{\Gamma(3/4)}{\Gamma(1/4)}
  =
  \int_{\partial B}
  f
  =
  \int_{-\infty}^\infty
  \frac{\log\Gamma(1+it)}{\cosh(2\pi t)}\; dt
  -
  \int_{-\infty}^\infty
  \frac{\log\Gamma(it)}{\cosh(2\pi t)}\; dt.
 \]
 Using $\Gamma(1+it)=it\Gamma(it)$ and taking real parts to avoid considerations about the branch of the logarithm,
 \[
  \log\frac{\Gamma(3/4)}{\Gamma(1/4)}
  =
  \int_{-\infty}^\infty
  \frac{\log |t|}{\cosh(2\pi t)}\; dt
  =
  \frac{1}{\pi}
  \int_{0}^\infty
  \frac{\log (t/2\pi)}{\cosh t}\; dt
  =
  I
  -
  \int_{0}^\infty
  \frac{\log (2\pi)}{\pi \cosh t}\; dt.
 \]
 The last integral is $\log\sqrt{2\pi}$ just changing $t=\log u$. 
\end{proof}

\begin{proof}[Proof of Theorem~\ref{main}]
 The identity $\theta(z)=\prod_{n=1}^\infty \big(1-e^{2\pi inz}\big)\big(1+e^{\pi i(2n-1)z}\big)^2$ follows from Lemma~\ref{triple} with $q=e^{\pi i z}$. Some elementary manipulations with the definition of $\eta$ show $\theta(z)=\eta^2\big(\frac 12 z+\frac 12\big)/\eta(z+1)$. 
 Let $Q=x^2+y^2$ and $Q'=2x^2-2xy+y^2$ with $z_Q=i$ and $z_{Q'}=\frac{1+i}{2}$. 
 We  have $\zeta(s,Q)=\zeta(s,Q')$ because $Q'=x^2+(x-y)^2$. 
 Then Proposition~\ref{kronecker} implies $|\eta(z_{Q'})|^2/|\eta(z_{Q})|^2=\sqrt{2}$
 and, noting
 $\theta(i)=|\theta(i)|$ and $|\eta(z+1)|=|\eta(z)|$,
 \[
  \theta(i)
  =
  \theta(z_Q)
  =
  \Big|
  \frac{\eta(z_{Q'})}{\eta(z_{Q}+1)}
  \Big|^2|\eta(z_{Q}+1)|
  =\sqrt{2}|\eta(z_Q)|.
 \]
 Recalling Lemma~\ref{integral}, Theorem~\ref{main} is equivalent to 
 $I=-\log\big(2|\eta(z_Q)|^2\big)$. 
 By Lemma~\ref{rn}, we have $\zeta(s, Q)=4\zeta(s)L(s)$ and, since Proposition~\ref{kronecker},  we must prove 
 \[
  \lim_{s\to 1^+}
  \Big(
  \frac{2}{\pi}\zeta(s)L(s)-
  \zeta(2s-1)
  \Big)
  =
  I.
 \]
 It is known $\zeta(s)\sim (s-1)^{-1}+\gamma$ as $s\to 1$ with $\gamma$ the Euler-Mascheroni constant, and it admits a short elementary proof \cite[p.23]{iwaniec}. 
 Using\footnote{For a quick proof write
 $\Gamma(s)=\lim_{n\to\infty}\int_0^n\big(1-\frac{x}{n}\big)^n x^{s-1}\; dx$ to obtain, by repeated partial integration,
 $\lim\frac{n! n^s}{s(s+1)\cdots (s+n)}$ (Gauss' definition of $\Gamma$). The derivative of its logarithm at $s=1$ gives finally
 $\Gamma'(1)=-\gamma=\lim\big(\log n -\frac{1}{1} -\frac{1}{2}- \cdots -\frac{1}{n}\big)$.}
 $\Gamma'(1)=-\gamma$ we have
 $\zeta(s)-2\Gamma(s)\zeta(2s-1)\to 0$ and the previous limit is 
 \[
  \lim_{s\to 1^+}
  \Big(
  \frac{4}{\pi}\Gamma(s)L(s)-1
  \Big)
  \zeta(2s-1)
  =
  \lim_{s\to 1^+}
  \frac{4\Gamma(s)L(s)-\pi}{2\pi(s-1)}
  =
  \frac{2}{\pi}
  \frac{d}{ds}\Big|_{s=1}\big(\Gamma(s)L(s)\big)
 \]
 by L'H\^opital's rule. It only remains to show that this derivative is $\pi I/2$. Plainly $\Gamma(s)n^{-s}=\int_0^\infty t^{s-1}e^{-nt}\; dt$. Then 
 \[
  \Gamma(s)L(s)=
  \int_0^\infty 
  t^{s-1}
  \big(
  e^{-t}
  -e^{-3t}
  +e^{-5t}
  -e^{-7t}+\dots
  \big)\; dt
  =
  \int_0^\infty 
  \frac{t^{s-1}}{2\cosh t}\; dt
 \]
 and Lemma~\ref{integral} implies the result differentiating under the integral sign.
\end{proof}


\subsection*{Acknowledgments}
I am deeply indebted to E. Valenti.


\

\textsc{Departamento de Matem\'aticas and ICMAT, Universidad Aut\'onoma de Madrid, 28049 Madrid, Spain}
 
\textit{Email address}: \url{fernando.chamizo@uam.es}

\end{document}